\documentclass{amsart}
\usepackage{mathtools,verbatim, amssymb, amsmath, amscd, color, enumitem}
\usepackage[utf8]{inputenc}
\usepackage{tikz}
\usepackage{tikz-cd}
\usepackage[colorlinks=true]{hyperref}

\newcommand{\B}[1]{\mathbf #1}

\newcommand{\wt}[1]{\widetilde #1}

\newtheorem{corollary}{Corollary}

\newtheorem{remark}{Remark}
\newtheorem{theorem}{Theorem}
\newtheorem{lemma}{Lemma}
\newtheorem{fact}{Fact}

\newtheorem{definition}{Definition}
\newtheorem{example}{Example}
\newtheorem*{question*}{Question}
\newtheorem*{conjectureg*}{Gromov Conjecture}
\newtheorem*{conjectured*}{Rationality Conjecture}
\begin{document}
\author[Marcinkowski]{Micha{\l} Marcinkowski}
\address{Uniwersytet Wroc{\l}awski \textbf{\textit{\&}} Instytut Matematyczny Polskiej Akademii Nauk, Warszawa}
	\email{marcinkow@math.uni.wroc.pl}
\thanks{The work on this paper was conducted during the author's internship at the Warsaw Center of Mathematics and Computer Science.
The author was supported by a scholarship of the Foundation for Polish Science and by the National Science Centre grant 2012/06/A/ST1/00259}
\title[Rationally inessential macroscopically large manifolds]
{Gromov positive scalar curvature conjecture and rationally inessential macroscopically large manifolds}
\maketitle
\begin{abstract}
We give the first examples of rationally inessential but macroscopically large manifolds.
Our manifolds are counterexamples to the Dranishnikov rationality conjecture. 
For some of them we prove that they do not admit a metric of positive scalar curvature, thus satisfy the Gromov positive scalar curvature conjecture. 
Fundamental groups of our manifolds are finite index subgroups of right angled Coxeter groups. 
The construction uses small covers of convex polyhedrons (or alternatively Davis complexes) and surgery.    
\end{abstract}

\section{Introduction} 

Let $X$ be a metric space and let $Y$ be a topological space. 
We say that a map $f \colon X \to Y$ is uniformly cobounded if there exists a real number $C$ such that $diam(f^{-1}(y)) < C$ for every $y \in Y$. 

\begin{definition}
The macroscopic dimension of $X$, denoted $dim_{mc}(X)$, is the smallest number $k$ such that there exist 
a $k$-dimensional simplicial complex $K$ and a continuous, uniformly cobounded map $f \colon X \to K$. 
\end{definition}

Let $M$ be a Riemannian manifold of topological dimension $n$, and let $\wt M$ be the universal cover of $M$ with the pullback Riemannian metric.
Note that since $\wt M$ can be given a structure of simplicial complex, $dim_{mc}(\wt M)$ is never greater than the topological dimension.

Macroscopic dimension was defined by Gromov (\cite{MR1389019}) in the search of topological obstructions 
for manifolds to admit a Riemannian metric with positive scalar curvature (briefly PSC). 
He conjectured that such manifolds tend to have deficiency of macroscopic dimension in the following sense

\begin{conjectureg*}
Let $M$ be a closed $n$-dimensional manifold. If $M$ admits a Riemannian metric of positive scalar curvature,
then $dim_{mc}(\wt M) \leq n-2$. 
\end{conjectureg*}

We always assume that the metric on $\wt M$ is a pullback of some Riemannian metric on $M$.
Macroscopic dimension of $\wt M$ does not depend on a metric chosen on $M$.

The $n-2$ in the conjecture comes from the following prototypical example:
for any $M^{n-2}$, the manifold $M' = M \times S^2$ admits a PSC metric.
We have $dim_{mc}(\wt M') = dim_{mc}(\wt M \times S^2) = dim_{mc}(\wt M) \leq n-2$.
Thus an inequality in the conjecture is sharp.

There is also a version of the Gromov Conjecture, called the weak Gromov conjecture,
which asserts that if $M$ admits a PSC metric, then $dim_{mc}(\wt M) \leq n-1$.\\

The Gromov conjecture was proven for $3$-dimensional manifolds (\cite{MR720933}) and for manifolds whose fundamental groups 
satisfy certain assumptions of analytical flavor (\cite{MR2578547,dranish.conjecture}).
In the present state of the art, the Gromov conjecture (and even its weak version) is considered to be out of reach.
It implies other longstanding conjectures, e.g.~the Gromov-Lawson conjecture, 
which asserts that no aspherical manifold admits a PSC metric.\\ 

An $n$-dimensional manifold $M$ is called \textbf{macroscopically large} if $dim_{mc}(\wt M) = n$.
Let us consider the following

\begin{example}\label{ex}
Let $M$ be a closed oriented manifold, $\pi = \pi_1(M)$, and let $B\pi$ be a classifying space endowed with a structure of a simplicial complex.
Denote by $f \colon M \to B\pi$ the map classifying the universal bundle. 
If $f_*([M])=0 \in H_n(B\pi;\B Z)$, then there is a homotopy of $f$ to some map $g \colon M \to B\pi^{[n-1]}$. 
It follows, that there exist an \textbf{equivariant homotopy} of a lift $\wt f \colon \wt M \to E\pi$ 
to $\wt g \colon \wt M \to E\pi^{[n-1]}$. Then $\wt g$ is a cobounded map, thus $M$ cannot be macroscopically large. 
\end{example}

One can ask if the property that a manifold $M$ is large or not can be expressed in homological terms.
To do that, let us introduce the following notions (we keep the notation from Example \ref{ex}). 
We call $M$ \textbf{inessential} if $f_*([M]) = 0 \in H_n(B\pi;\B Z)$ and \textbf{rationally inessential} if $f_*([M]) = 0 \in H_n(B\pi;\B Q)$.
Note that $M$ is rationally inessential if and only if $f_*([M]) \in H_n(B\pi;\B Z)$ is torsion.
An example of rationally inessential but essential orientable manifold is $\B R \B P^3$. 
Obviously $dim_{mc}(\widetilde{\B R \B P}{}^3) = 0$, thus being essential is not enough to be macroscopically large. 
Gromov expected, that if $M$ is rationally essential, then $M$ is macroscopically large.
A. Dranishnikov in \cite{dranish.nowa} disproved this conjecture and found the right homology theory where one should place the fundamental class $[M]$
to test if $M$ is large just by checking if the class is non-trivial. 
Moreover, it is showed that $[M]$ is large if and only if there exist a \textbf{bounded homotopy} from $\wt f \colon \wt M \to E\pi$
to some map which ranges in $E\pi^{[n-1]}$ 
(these results are described in more details in \ref{hom.car}).
In \cite{dranish.conjecture} it is conjectured that

\begin{conjectured*}
If $M$ is rationally inessential, then it is macroscopically small. 
\end{conjectured*}

It would imply the weak Gromov conjecture for rationally inessential manifolds.
In the paper we give counterexamples to this conjecture.
In terms of homotopy theory, they are rationally inessential manifolds such that $\wt f \colon \wt M \to E\pi$
can not be deformed by means of bounded homotopy to a map which ranges in $E\pi^{[n-1]}$.   
In the case when our manifolds are spin, we prove that they do not admit a PSC metrics.
Thus they satisfy the Gromov Conjecture. If they are not spin, the conjecture is open.\\ 

\textbf{Outline of the construction}. Let $K$ be an $n$-dimensional simple convex polyhedron, e.g.~$n$-dimensional cube
\footnote{By $n$-dimensional polyhedron we mean an intersection of finite number of half-spaces in the $n$-dimensional Euclidean space.
A polyhedron $K$ is simple if a neighborhood of every vertex of $K$ looks like a neighborhood of a vertex in the $n$-dimensional symplex.}.
Assume that each maximal face of $K$ is colored by one of $n$ colors such that every pair of different non-disjoint faces have different colors \footnote{Due to Lemma \ref{passing_to_sub} this assumption is not restrictive at all.}.
To construct a manifold $N$ out of $K$ we use 'the reflection trick'.
That is, we glue up $2^n$ copies of $K$ along maximal faces. 
The way of how we glue them depends of the coloring of faces.
To obtain a counterexample $M$ to the Dranishnikov conjecture we attaching a bunch of handles and 'fill up' some loops in the connected sum $N \# N$. 
This is made such that $M$ is rationally inessential and $\pi(M)$ is a finite index, torsion free subgroup of a Coxeter group.
These properties are crucial in proving that $M$ is macroscopically large and, if $M$ is spin, does not admit a PSC metric.  

\section{Preliminaries}

\subsection{Homological characterisation of macroscopically large manifolds}\label{hom.car}

A. Dranishnikov gave (\cite{dranish.nowa}) a homological criterion one can use to detect if $M$ is macroscopically large.
We briefly discuss his result in the form we are going to use it. 

Let $X$ be a locally finite simplicial complex. Let $C^{lf}_*(X;\B Z)$ be the module of $\B Z$-valued simplicial chains on $X$.
Here chains need not to be finitely supported nor bounded. 
The chain complex $(C^{lf}_*(X; \B Z),\partial)$ with the standard differencial 
defines the locally finite homology groups $H^{lf}_*(X,\B Z)$ (see \cite[Ch.11]{MR2365352} for an exhaustive treatment of the locally finite homology). 
If $A < X$ is a subcomplex of $X$, the notion of relative locally finite homology is defined as usual by the quotient chain complex $C^{lf}_*(X; \B Z) / C^{lf}_*(A; \B Z)$. 
In \cite{dranish.nowa} a more general definition (with an arbitrary coefficients) of coarsely equivariant homology is given. 
For $\B Z$ coefficients, the coarsely equivariant homology is naturally isomorphic to the locally finite homology.

Let $\wt X$ be the universal cover of $X$ with the induced simplicial structure and let $\pi = \pi_1(X)$.
Recall that $H_*(X;\B Z) = H_*^{\pi}(\wt X;\B Z)$, where the last group is defined by means of $\pi$-equivariant chains $C_*^{\pi}(\wt X;\B Z)$.
The inclusion $i \colon C_*^{\pi}(\wt X;\B Z) \to C^{lf}_*(\wt X;\B Z)$ induces the so called equivariant coarsening map 
$ec_* \colon H_*(X;\B Z) \to H^{lf}_*(\wt X;\B Z)$.

\begin{theorem}{\cite[Th.2.2 and Th.4.5]{dranish.nowa}}\label{dranish}
Let $M$ be an $n$-dimensional, oriented, closed manifold and let $\Gamma = \pi_1(M)$. 
Suppose that $B\Gamma$ is realised by a finite simplicial complex.
Let $f \colon M \to B\Gamma$ be a map classifying the universal cover.
\begin{enumerate}
\item 
 Then $M$ is macroscopically large if and only if $ec_nf([M]) \neq 0 \in H^{lf}_n(E\Gamma,\B Z)$, where $E\Gamma = \widetilde{B\Gamma}$.
\item  
We call a homotopy $H \colon \wt M \times I \to E\Gamma$ bounded, if for some number $C$ and every $x \in \wt M$ we have $diam(H[{x} \times I]) < C$.
On $E\Gamma$ we consider any proper geodesic metric.
Then $M$ is macroscopically large if and only if there is no bounded homotopy from $f$ to a map
$g \colon \wt M \to E\Gamma^{[n-1]}$.
\end{enumerate}
\end{theorem}

\begin{remark}\normalfont
There is another notion of macroscopically large manifolds given by Gong and Yu. 
It is expressed in terms of non-vanishing of the fundamental class in the coarse homology group $HX_*(\wt M,\B Q)$ 
(\cite[Def.8.2.2]{MR2986138}). As it is shown in \cite[Th. 4.2]{dranish.conjecture}, this definition is equivalent to ours provided that the coefficient module is taken to be $\B Z$. 
\end{remark}

\subsection{Small covers}

The idea of a small cover of a simplicial complex is the main ingredient of the construction.
They were investigated in the seminal paper of M.~Davis and T.~Januszkiewicz (\cite{DJ}).
Here we discuss the notion of a small cover and collect some facts we use later. 

\subsubsection{Basic definitions}\label{basicdef}
Let $L$ be an $n$-dimensional simplicial complex. 
By $L^b$ we denote its barycentric subdivision. 
By definition, it is the geometric realisation of the poset of nonempty simplices of $L$.
It means that every $(l-1)$-dimensional simplex $\tau \in L^b$ is given by a chain $\tau = (\sigma_1 < \ldots < \sigma_l)$, $\sigma_i \in L$.  
Let $\sigma < L$ be a simplex. We define the \textbf{face} $F_{\sigma}$ to be the geometric realisation of the poset
$\{\sigma' \in L~|~\sigma \leq \sigma' < L\}$. 

Thus $F_{\sigma}$ is a subcomplex of $L^b$. 
If $\sigma$ is $k$-dimensional, then $F_\sigma$ is $n-k$ dimensional.
If $\sigma = [v_0,\ldots,v_k]$, then $F_{\sigma} = F_{v_0} \cap \ldots \cap F_{v_k}$. 
The set of faces introduces on $L$ a so called mirror structure. 
We use the notation $L^*$ if we refer to $L$ with this mirror structure. 
Note that if $L$ is not a manifold, then a face need not be homeomorphic to the disc.

Let $\B G$ be a $\B Z_2$-linear space and let $\lambda \colon L^{[0]} \to \B G$.  
The function $\lambda$ is defined on the set of vertices of $L$; equivalently, it is defined on the $n$-dimensional faces of $L^*$. 
Let $F$ be an $(n-k)$-dimensional face, then $F = F_{v_0} \cap \ldots \cap F_{v_k}$ (i.e. $F$ is dual to $[v_0, \ldots, v_k]$, 
for the unique set of vertices $v_i$.
We define $\B G_F = Span_{\B G} \langle \lambda(F_{v_0}),\ldots,\lambda(F_{v_k}) \rangle$. 
We call $\lambda$ a \textbf{characteristic function} if for each $(n-k)$-dimensional face,
linear dimension of $\B G_F$ equals $k+1$. 

Now we define a 'cover space' associated to $L^*$ and a characteristic function $\lambda$. 
Let $p \in L^*$. By $F(p)$ denote the minimal face which contains $p$. 
Let $C(L^*)$ be the cone over $L^*$. Note that $C(L^*)$ contains $L^*$ as the base of the cone.
Consider the space $M_L = C(L^*) \times \B G / \sim_{\lambda}$, 
where $p \times g \sim_{\lambda} p' \times g'$ if and only if $p = p' \in L^*$ and $g'-g \in \B G_{F(p)}$. 

In this construction the copies of $C(L^*)$ are glued along some $n$-dimensional faces, according to function $\lambda$.
Note that on $M_L$ we have a natural $\B G$-action, which induces a quotient map $p_L \colon M_L \to C(L^*)$.
If the linear dimension of $\B G$ is equal to $n+1$, we call $M_L$ a \textbf{small cover} of $C(L^*)$. 
For examples we advise to consult \cite[Example 1.19]{DJ}. 

Now we introduce a simplicial structure on $M_L$.
First we put on $C(L^*)$ a simplicial structure given by $L^b$. 
Since each face is a subcomplex of $L^b$, the gluings between copies of $C(L^*)$ are made along subcomplexes of this triangulations. 
Thus the simplicial structures on $C(L^*)$'s carry to $M_L$. 

In the sequel, we will need a particular case of this construction.
Let $\lambda \colon L^{[0]} \to \B G$ be a characteristic function and let $dim(\B G) = n+1$.
Let $e_i$, $i = 0 \ldots n$,  be a basis of $\B G$. Later we simply write $\B G = \B Z_2^{n+1}$. 
We call $\lambda$ \textbf{folding on a simplex} if $\lambda(v) = e_i$ for some $i$ and every $v \in L^{[0]}$. 
The name comes from the fact that such a $\lambda$ defines a map $f_{\lambda} \colon L \to \Delta^n$, where  $\Delta$ is the $n$ dimensional simplex.
Indeed, if we think of $\Delta$ as a simplex spanned by the standard vectors $e_i$ in $\B R^{n+1}$, then for every $v \in L^{[0]}$
we define $f_{\lambda}(v) = \lambda(v)$ and extend $f$ linearly to the whole of $L$. 

Note that in this case being characteristic means that $\lambda(v) \neq \lambda(w)$ if $v$ and $w$ are incident.

We end this section with the following lemma.

\begin{lemma}\label{passing_to_sub}
Let $L$ be an $n$-dimensional complex. 
There exists a folding on a simplex characteristic function for $L^b$.
Thus, having an arbitrary complex, we can always construct a folding on a simplex
characteristic function after passing to the barycentric subdivision.
\end{lemma}
\begin{proof}
Every vertex $v \in L^b$ is a chain of length $1$. 
Assume that $v = (\sigma)$ and $\sigma$ is $i$-dimensional. 
We can put $\lambda(v) = e_i$. 
Indeed, if $v_1 = (\sigma_1)$ and $v_2 = (\sigma_2)$ are connected by an edge $e$, 
then $e = (\sigma_1<\sigma_2)$ or $e = (\sigma_2<\sigma_1)$. 
Thus $\sigma_1$ and $\sigma_2$ have different dimensions. 
\end{proof}

\subsubsection{Properties}\label{properties}

Let $L$ be a simplicial complex. The right angled Coxeter group $W_L$ associated to $L$ is given by the presentation 
$$
W_L = \langle L^{[0]}|\thinspace v^2,[v,w]\text{ for }v,w\in L^{[0]}\text{ and }(v,w)\in L^{[1]}\rangle.
$$

For any Coxeter group $W_L$ there exist a simplicial complex $\Sigma_{W_L}$, called the Davis complex of $W_L$, 
with a proper, cocompact action of $W_L$. 
The fundamental domain of the $W_L$-action on $\Sigma_{W_L}$ is simplicialy isomorphic to $C(L)$. 
For infinite $W_L$, the complex $\Sigma_W$ is contractible (see \cite[ch.7]{Davis}). 

If $\lambda$ is any function from $L^{[0]}$ to $\B G$,
then $\Lambda$ uniquely extends to a homomorphism from $W_L$ to $\B G$. 

\begin{fact}\label{factsmall}
Let $M_L$ be a cover associated to some characteristic function $\lambda$. Then
\begin{enumerate} 
	\item $\pi_1(M_L) = ker(\Lambda)$ is a torsion free finite index subgroup of $W_L$. 
	\item $M_L$ is aspherical and its universal cover is homeomorphic to $\Sigma_{W_L}$.
	\item If $M_S$ is a small cover associated to a sphere $S$, then it is an oriented manifold. 
\end{enumerate}
\end{fact}

\begin{proof}

\begin{enumerate}[leftmargin=*]
\item The fact that $\pi_1(M_L) = ker(\Lambda)$ is proved in \cite[Col. 4.5]{DJ}. 
To prove that $\pi_1(M_L)$ is torsion free, assume that $g \in \pi_1(M_L)$ is a nontrivial element of finite order. 
For $T \subset L^{[0]}$ define $W_T$ to be the subgroup of $W_L$ generated by $T$. 
In \cite[Lemma 1.3]{D} it is showed that every finite subgroup of $W_L$ is conjugate to a subgroup of finite $W_T$, for some $T$.  
Since $W_T$ is finite, it follows that $T$ spans a simplex in $L$. 
Indeed, otherwise $W_T$ would contain infinite dihedral group generated by non adjacent vertices. 
The element $g$ generates a finite subgroup of $W_L$ and, thus  
there exist a simplex $\sigma \in L$ such that $g$, up to conjugation, can be written in generators corresponding to
the vertices of $\sigma$. Since $\Lambda(g) = 0$, each of the generators has to appear even number of times. 
All the generators we used to express $g$ pairwise commutes. Thus $g = e$, which gives a contradiction. 
\item This is  \cite[Lemma 4.4]{DJ}.
\item This is \cite[Prop. 1.7]{DJ}.
\end{enumerate}  

\end{proof}

\begin{lemma}\label{G}
Let $h \colon L_1 \to L_2$ be a simplicial map such that $k$-simplices are mapped to $k$-simplices for all $k$. 
Let $\B G_1 < \B G_2$ be a linear inclusion of $\B Z_2$-linear spaces and 
let $\lambda_{L_i} \colon L_i^{[0]} \to \B G_i$ be characteristic functions such that $\lambda_{L_1} = \lambda_{L_2}h$.
Then $G_{F(p)} = G_{F(h(p))}$ for every $p \in L_1$.
\end{lemma}

\begin{proof}
Let $\tau_p$ be the minimal simplex in $L_1^b$ containing $p$ and let $\tau_{h(p)}$ be the minimal simplex in $L_2^b$ containing $h(p)$.
Notice that $h(\tau_p) = \tau_{h(p)}$. 
Let $\tau_p = (\sigma_1 < \ldots < \sigma_l)$ and $\sigma_1 = [v_1, \ldots, v_s]$.
Since $\tau_p$ is minimal, it is contained in each face $F_\sigma$ which contains $p$. 
Thus $\sigma < \sigma_1$ and $F_{\sigma_1} \subset F_\sigma$.
It means that $F_{\sigma_1}$ is the minimal face containing $p$ 
and $G_{F(p)} = Span_{G_1}\langle \lambda_{L_1}(v_1) , \ldots , \lambda_{L_1}(v_s) \rangle$. 
The same is true in $L_2$; i.e.: $\tau_{h(p)} = h(\tau_p) = (h(\sigma_1) < \ldots < h(\sigma_l))$, where $h(\sigma_1) = [h(v_1), \ldots, h(v_s)]$ 
and $G_{F(h(p))} = Span_{G_2} \langle \lambda_{L_2}h(v_1), \ldots, \lambda_{L_2}h(v_s) \rangle$.
Since $\lambda_{L_1} = \lambda_{L_2}h$, we have that $G_{F(h(p))} = G_{F(p)}$.  
\end{proof}

Let us introduce the following notation.
Recall that  $C(X) = X \times I /_{\sim_X}$, 
where $I$ is the closed interval and $(x,t) \sim_X (y,s)$ if and only if $t=s=1$.
Given a map $h \colon A \to B$, we define $C_h \colon C(A) \to C(B)$ by the formula
$C_h(a \times t /_{\sim_A}) = h(a) \times t /_{\sim_B}$. 

\begin{corollary}\label{inclusion}
We use the notation from Lemma \ref{G}. Let $w \in \B G_2$ and let $C_h \colon C(L_1) \to C(L_2)$ be the map induced from $h$ to the cones. 
The function $M_{h,w} \colon M_{L_1} \to M_{L_2}$ given by $M_{h,w}(p \times v/_{\sim_{\lambda_1}}) = C_h(p) \times (v+w)/_{\sim_{\lambda_2}}$ is well defined.
If $h$ is injective, then $M_{h,w}$ is injective. 
\end{corollary}

\begin{proof}

Since $M_{h,w}$ is the composition of the action of $w$ and $M_{h,0}$, it is enough to prove Corollary for $M_{h,0}$. 
The fact that $M_{h,0}$ is well defined follows from Lemma \ref{G}. 
Assume that $h$ is injective. Take two different points $x_1,x_2 \in M_{L_1}$,  
let $x_i = p_i \times g_i/_{\lambda_1}$. We may assume that $p_i \in L_1 < C(L_1)$, otherwise injectivity is trivial.
If $p_1 \neq p_2$, then $h(p_1) \neq h(p_2)$ and $M_{h,0}(x_1) \neq M_{h,0}(x_2)$.
If $p = p_1 = p_2$, then $g_1g_2^{-1} \notin G_{F(p)} = G_{F(h(p))}$ and $M_{h,0}(x_1) \neq M_{h,0}(x_2)$. 
\end{proof}

Let $L^n$ be a simplicial complex and $M_L$ its small cover associated to $\lambda \colon L^{[0]} \to \B Z_2^{n+1}$. 
If $g \in \B Z_2^{n+1}$, then by $|g|$ we denote the number of nonzero coordinates of $|g|$. 
Let $c$ be a simplicial chain in the barycentric subdivision $L^b$, i.e. $c$ is a formal sum of simplices of $L^b$. 
Note that $C(c)$, the cone over $c$, is itself a chain in $C(L^b)$.  
We define the \textbf{lift} of $C(c)$ to $M_L$ by
$M_c = \sum_{g \in \B Z_2^{n+1}} (-1)^{|g|}(C(c) \times g)$.

\begin{lemma}\label{lift}
Let $c$ be a chain in $L^b$ and let $\sigma \in L^b \times g < M_L$ for some $g \in \B Z_2^{n+1}$.
Then $\sigma$ does not appear in $\partial M_c$.
\end{lemma}

\begin{proof}
Note that $\partial M_c =  \sum_{g \in \B Z_2^{n+1}} (-1)^{|g|}(\partial C(c) \times g)$.
So the Lemma is clear if $\sigma$ does not appear in $\partial C(c) \times g$. 
Assume on the other hand that it appears in $\partial C(c) \times g$ and $F$ is the smallest face containing $\sigma$. 
By the construction $Stab(\sigma) = \lambda(F)$ and $\sigma$ is glued exactly with the copies of $\sigma$ contained in $\partial C(c) \times g'$,
where $g' = g + x$ and $x \in \lambda(F)$. 
Thus in $M_c$, $\sigma$ appears $|\lambda(F)| = 2^k$ times (for some $k$), the same number of times with the sign plus and minus. 
\end{proof}

\subsection{Surgery}

Let $M$ be an $(n>3)$-dimensional manifold. 
The boundaries of $H_i = S^i \times D^{n-i}$ and $H_{n-i-1} = D^{i+1} \times S^{n-i-1}$ are homeomorphic and equal to $S^i \times S^{n-i-1}$. 
Thus every embedding (called framing) $f \colon H_i \to M$ defines also an embedding $\partial f \colon \partial H_{n-i-1} \to M$.
We consider a manifold $M' = M \backslash f(H_i) \cup_{\partial f} H_{n-i-1}$ where $\partial H_{n-i-1}$ is glued to 
$\partial(M \backslash f(H_i))$ by $\partial f$. 
This procedure is called a surgery of index $i+1$. 
We present sketch of a proof of the following classical lemma.

\begin{lemma}\label{surgery}
Let $X$ be a topological space and let $M$ be a compact, oriented, $n$-dimensional manifold, $n > 3$.
Assume that $\Gamma_X = \pi_1(X)$ is finitely generated. 
Then for every map $f_M \colon M \to X$ there exists a sequence of surgeries of index $1$ and $2$, which results in a manifold $M'$ such that:
a) there exist a map $f_{M'} \colon M' \to X$ such that $\pi_1(f_{M'})$ is an isomorphism, b) $f_M([M])$ = $f_{M'}([M'])$. 
\end{lemma}

\begin{proof}

First we modify the map $f_M$ to be an epimorphism. 
For this we use surgery of index $1$, that is we attach handles. 
Let $\gamma_i$ be a set of loops in $X$ which represents generators of $\Gamma_X$.  
For each $\gamma_i$ we attach a handle to $M$. We call the new manifold $M_0$. 
We can extend $f_M$ to $f_{M_0}$ such that if we take a path which goes along the handle and connects its ends, it is mapped to $\gamma_i$. 
The homology class of the image does not change and $f_{M_0}$ is an epimorphism. 
Now we need to fill up some loops which are in $N = ker(\pi_1(f_{M_0}))$. 
The subgroup $N$ is normally finitely generated. 
Let ${\eta_i}$ be a set of normal generators of $N$. 
Then we can perform a surgery of index $2$ on $M_0$ along these loops, obtaining a manifold $M'$. 
Since images of loops are contractible in $X$, the map $f_{M_0}$ can be extended to a map $f_{M'}$.
Moreover, the homology class does not change and $f_{M'}$ is an isomoprhism. 

\end{proof}
 
\section{Counterexamples}\label{counterexamples}

\subsection{The construction}\label{construction}

In this section we define manifolds $M^n_k$. It is done in three steps.  

\subsubsection{Step 1: The complex $L^n_k$}

Let $D_0$ be an oriented $(n>3)$-dimensional closed disk. 
Let $D_1,\ldots,D_{k-1} \subset D_0$ be a collection of pairwise disjoint subspaces of $D_0$ homeomorphic to closed $n$-discs. 
We define a complex $L = L^n_k$ to be the following space: start with $D_0 \backslash (intD_1 \cup \ldots \cup intD_{k-1})$ 
and glue each boundary $\partial D_i$, $i > 0$, with $\partial D_0$ by orientation preserving homeomorphisms. 
Note that we have an inclusion $i \colon \partial D_0 \to L^n_k$. 
The space $L^n_k$ is a manifold except the singular sphere $S = i(\partial D_0)$ where we have a ramification
of degree $k$. On $S$ we have the orientation induced from $\partial D_0$. 
 
\subsubsection{Step 2: Small covers}  
We pick a triangulation on $L = L^n_k$ and assume that it admits a folding on a simplex characteristic function 
$\lambda \colon L^{[0]} \to \B Z_2^{n+1}$. 
By $L^*$ we denote the complex dual to $L$. 
Moreover, assume that the restriction of $\lambda$ to vertices of $S$ is again a folding on a simplex characteristic function
for the complex $S$, which ranges in the subspace spanned by the first $n$ generators of $\B Z_2^{n+1}$.
Such a triangulation exists. Indeed, having any triangulation on $L$ we can pass to the barycentric subdivision 
and use the characteristic function defined in the proof of Lemma \ref{passing_to_sub}. It satisfies the above assumptions. 

Let $p_L \colon M_L \to C(L^*)$ be the small cover defined by $\lambda$.
Let $p_S \colon M_S \to C(S^*)$ be the small cover defined by $\lambda_S$, where $\lambda_S$ is the restriction of $\lambda$ to the vertices of $S$.
Note that $M_S$ is an oriented manifold.	 
By Corollary \ref{inclusion}, in the complex $M_L$ we see two copies of $M_S$. Namely they are
$$
M_S^0 =  C(S^*) \times (\B Z_2^n \times {0})/_{\sim_0} < M_L = C(L^*) \times \B Z_2^{n+1}/_{\sim_{\lambda}}
$$
$$
M_S^1 =  C(S^*) \times (\B Z_2^n \times {1})/_{\sim_1} < M_L = C(L^*) \times \B Z_2^{n+1}/_{\sim_{\lambda}}. 
$$

By the relation $\sim_i$ we mean the relation $\sim_{\lambda}$ restricted to $M_S^{i}$.
It coincides with $\sim_{\lambda_S}$ under the obvious identification with $M_S$.
The chain $M_S^{01} = M_S^0-M_S^1$ is an $n$-dimensional cycle, thus defines a class $[M_S^{01}] \in H_n(M_L;\B Z)$.
The signs were chosen so that $M_S^{01}$ is the lift of $C(S^b)$ by $p_L$, as described in the discussion before Lemma \ref{lift}.
We define an oriented manifold  $N = M_S^0 \#(-M_S^1)$. Since $M_L$ is connected, there exists a map  $f_N \colon N \to M_L$
such that the pushforward of the fundamental class equals $[M_S^{01}]$. 
 
\subsubsection{Step 3: Surgery}

Let $M^n_k$ and $f_{M^n_k} \colon M^n_k \to M_L$ be an $n$-dimensional 
manifold and a map obtained by the procedure described in Lemma \ref{surgery} applied to 
$N$ and $f_N$. From Fact \ref{factsmall}(2) we know that $M_L$ is aspherical, thus $f_{M^n_k}$ classifies the universal cover. 
Moreover we have that the pushforward of the fundamental class of $M^n_k$ equals $[M_S^{01}]$.   
   
\subsection{$M^n_k$ is a counterexample to the Dranishnikov rationality conjecture}

\begin{lemma}\label{mv}
	Let $S$ be as in the construction (Step 1). Then $H_{n-1}(L^n_k;\B Z) = \B Z_k$ and $[S]$ is a generator. 
\end{lemma}

\begin{proof}
It follows from the Mayer-Vietors exact sequence. 
Let $D = D_0 \backslash (intD_1 \cup \ldots \cup intD_{k-1})$. We have a quotient map
$$
q \colon D \to L,
$$
which glues the boundaries of $\partial D_i$ with $\partial D_0$ as in the construction.
Let $D_k$ be another $n$-disk embedded in the interior of $D$.
Let $L_{\bullet} = L \backslash q(intD_k)$. 
To use the Mayer-Vietoris sequence we decompose $L$ as follows: $L = L_{\bullet} \cup q(D_k)$.

We claim that $L_{\bullet}$ is homotopy equivalent to the wedge of $S^{n-1}$ and $k-1$ circles.
Indeed, let $\gamma_i$, $i=1,\ldots,k-1$, be a collection of disjoint arcs in $D \backslash intD_k$. 
Assume that each $\gamma_i$ connects a point on $\partial D_i$ with a point on $\partial D_0$. 
The subspace $X = \partial D_0 \cup (\partial D_1 \cup \gamma_1) \cup \ldots \cup (\partial D_{k-1} \cup \gamma_{k-1})$
is a deformation retract of $D \backslash intD_k$.
We can imagine that we start to inflate the disk $D_k$ such that at the end it fills up all the space between $X$. 
It is easy to see that the space $q(X)$ is homotopy equivalent to the wedge of $S^{n-1}$ and $k-1$ circles.
Moreover, the retraction of $D \backslash intD_k$ to $X$ carries down to a retraction 
of $L_{\bullet}$ to $q(X)$. This proves the claim.    

Note that $L_{\bullet} \cap q(D_k) = q(\partial D_k) \cong S^{n-1}$. 
The Mayer-Vietoris exact sequence reads:
$$
\begin{tikzcd}
H_{n-1}(q(\partial D_k)) \arrow{r}{i} & 
H_{n-1}(q(D_k)) \oplus H_{n-1}(L_{\bullet}) \arrow{r}{s} &
H_{n-1}(L) \arrow{d}{\delta} \\
& & H_{n-2}(q(\partial D_k)) 
\end{tikzcd}
$$
Thus $s$ is an epimorphism. 
Let us take a closer look at $i$.
Since $L_{\bullet} \cong S^{n-1} \vee S^1 \vee \ldots \vee S^1$, we have that $H_{n-1}(q(D_k)) \oplus H_{n-1}(L_{\bullet}) = \B Z$. 
This group is generated by $[S] = q_*[\partial D_0]$. 
Note that we can choose an orientation of $D_k$ such that the following holds:
$$
\partial(D \backslash D_k) = \partial D_0 + \partial D_1 + \ldots + \partial D_{k-1} - \partial D_k.
$$
After applying $q$ to this equation we have that $k[S] - q_*[\partial D_k] = 0$ in $H_{n-1}(L_{\bullet})$.
To finish, we note that $i([\partial D_k]) = q_*[\partial D_k]$. 
It follows that $i$ is the multiplication by $k$, thus the Lemma. 

\end{proof}

Now we are ready to prove the crucial lemma. The equivariant coarsening map 
$ec_* \colon H_*(X;\B Z) \to H^{lf}_*(\wt X;\B Z)$ and the locally finite homology $H^{lf}$ were defined in section \ref{hom.car}.

\begin{lemma}\label{S}
	Let $[M_S^{01}] \in H_n(M_L;\B Z)$ be as in the construction (Step 2). Then 
\begin{enumerate}
	\item $k[M_S^{01}] = 0 \in H_n(M_L;\B Z)$.
	\item $ec_n([M_S^{01}]) \neq 0 \in H^{lf}_n(\wt{M_L};\B Z)$. 
\end{enumerate}
\end{lemma}

\begin{proof}

\begin{enumerate}[leftmargin=*]

\item In Lemma \ref{mv} we defined $D = D_0 \backslash (intD_1 \cup \ldots \cup intD_{k-1})$ and the quotient map
$$
q \colon D \to L.
$$
On $D$ we consider the pullback simplicial structure.  
Let $\lambda_D$ be a folding on a simplex characteristic function for $D$ defined by 
$\lambda_D = \lambda q$. Let $p_D \colon M_D \to C(D^*)$ be the small cover associated to $\lambda_D$.   
By Corollary \ref{inclusion}, the map $q$ lifts to
$$
M_q \colon M_D \to M_L.
$$
We recall the definition of this map: $M_q(p \times x/_{\sim_{\lambda_D}}) = C_q(p) \times x/_{\sim_{\lambda}}$, 
for $p \in C(D), x \in \B Z_2^{n+1}$ and $C_q \colon C(D) \to C(L)$ is the map of cones induced by $q$.
The complex $D$ is oriented. The orientation defines an $n$-dimensional chain representing the relative to the boundary
fundamental class of $D$. 
Abusing notation we call it $D$ as well.
Note that $q_*(\partial D) = kS$. 
Let $M_D$ be the lift of the chain $D$ (for the definition see the discussion before Lemma \ref{lift}).

To simplify the notation we write $C(D)$ for $ C(D) \times 0$.
Thus $g.C(D) = C(D) \times g$, where by $g.C(D)$ we denote the action of $g \in \B Z_2^{n+1}$ on $C(D)$.
We have
\\

\begin{tabular}{ p{0.3cm} r l}
	&$q_*(\partial M_D)$ & $= \hspace{-.3cm} \sum\limits_{ \hspace{.3cm} g \in \B Z_2^{n+1}} \hspace{-.4cm} (-1)^{|g|}q_*(\partial g.C(D))$ \\
&&$= \hspace{-.3cm} \sum\limits_{\hspace{.3cm} g \in \B Z_2^{n+1}} \hspace{-.4cm} (-1)^{|g|}q_*(g.C(\partial D) +g. D)$ \\
&&$= \hspace{-.3cm} \sum\limits_{\hspace{.3cm} g \in \B Z_2^{n+1}} \hspace{-.4cm} (-1)^{|g|}(g.C(q_*\partial D) +g.(q_*D))$ \\
&&$= \hspace{-.3cm} \sum\limits_{\hspace{.3cm} g \in \B Z_2^{n+1}} \hspace{-.4cm} (-1)^{|g|}g.(kC(S)) +  \hspace{-.5cm} \sum
\limits_{\hspace{.3cm} g \in \B Z_2^{n+1}} \hspace{-.4cm} (-1)^{|g|}g.(q_*D)$\\
&&$= kM_S^{01}$. 
\end{tabular}\\
\\

The last equality follows from the following: 
the class $M^{01}_S$ is the lift of $C(S)$ by $\lambda_S$, thus equals $\sum (-1)^{|g|}g.C(S)$. 
Hence, if we show that $\sum (-1)^{|g|}g.D = 0$, we are done. 
Note that $D = \partial C(D) - C(\partial D)$. If we sum over $\B Z_2^{n+1}$ action, we have 
\\

\begin{center}
  \begin{tabular}{c}
 $\hspace{-.3cm} \sum\limits_{\hspace{.3cm} g \in \B Z_2^{n+1}} \hspace{-.4cm} (-1)^{|g|}g.D = 
 \partial M_D - \hspace{-.5cm} \sum\limits_{\hspace{.3cm} g \in \B Z_2^{n+1}} \hspace{-.4cm}  (-1)^{|g|} g.C(\partial D) = 0$.
\end{tabular}
\end{center}

Indeed, if $\sigma < g.D$ for some $g$, then $\sigma$ does not appear in $\sum (-1)^{|g|}g.C(\partial D)$ 
(every simplex of $g.C(\partial D)$ lies inside the cone $g.C(D)$) nor in $\partial M_D$ (Lemma \ref{lift}).
On the other hand, if $\sigma \nless g.D$ for every $g$, then it does not appear in $\sum (-1)^{|g|}g.D$. 
\\  
\item Consider the map
$$
\begin{tikzcd}
	H_*(M_L;\B Z) \arrow{r}{ec_*} &H^{lf}_*(\widetilde{M_L};\B Z).
\end{tikzcd}
$$
Let $\wt C$ be a lift of $C(L) \times 0 < M_L$ to the universal cover $\widetilde{M_L}$. By $\wt L$ we denote the base of the cone $\wt C$.
We remark that $\widetilde{M_L}$ is the Davis complex for $W_L$ and $\wt C$ is a fundamental domain of the $W_L$ action.
Let $int \wt C = \wt C - \wt L$.

Let $\rho$ be the homomorphism induced by the quotient map
$$
\rho_* \colon H_*^{lf}(\widetilde{M_L};\B Z) \to H_*^{lf}(\widetilde{M_L}, (int \wt C)^c; \B Z).
$$
Since $\wt C$ is a finite complex, it is straight-forward to prove that 
the locally finite homology of a pair $(\widetilde{M_L} , (int \wt C)^c)$ is isomorphic to the standard homology. 

The following equality holds by the excision theorem
$$
H_*(\widetilde{M_L} , (int \wt C)^c; \B Z) = H_*(\wt C,\wt L; \B Z).
$$
Of course 
$$
H_*(\wt C,\wt L; \B Z) = H_*(C(L),L; \B Z).
$$
Thus we can write that (compare with \cite{dlf}, where such maps were used to compute the locally finite homology of Coxeter groups)
$$
\rho_* \colon H^{lf}_*(\widetilde{M_L};\B Z) \to H_*(C(L),L; \B Z).
$$

From the definition of the comparision maps we see that
$\rho_nec_n[M_S^{01}] = [C(S)] \in  H_n(C(L),L; \B Z)$.
By the long exact sequence of the pair, the boundary map
$\delta \colon H_n(C(L),L; \B Z) \to H_{n-1}(L; \B Z)$ is an isomorphism. Moreover, $\delta([C(S)]) = [S]$. 
By Lemma \ref{mv}, the class $[S]$ is nonzero in $H_{n-1}(L;\B Z)$. Thus the class $ec_n([M_S^{01}])$ is nonzero. 
\end{enumerate}
\end{proof}

\begin{remark}\normalfont
Let $H^{ae}_*(M_L;\B Z)$ be the almost equivariant homology defined in \cite{dranish.nowa}. Let $ae_* \colon H_*(M_L;\B Z) \to H^{ae}_*(M_L;\B Z)$.
For $\B Z$ coefficients, $ae$ homology is isomorphic to the Block-Weinberger uniformly finite homology of $\pi_1(M_L)$, defined in \cite{MR1145337}.
The map $ec_*$ factors through $H^{ae}_*(M_L;\B Z)$, hence $ae_n(M^{01}_S)$ is nontrivial and torsion. 
To our knowledge, it is the first example of torsion class in the uniformly finite homology of a group. 
\end{remark}

Now we can prove that $M_k^n$ are counterexamples to the Rationality Conjecture.

\begin{theorem}
Manifolds $M^n_k$ are macroscopically large and rationally inessential.
\end{theorem}

\begin{proof}

By Step 3 of the construction, we have that  $(f_{M^n_k})_*[M^n_k] = [M_S^{01}]$. 
From Lemma \ref{S}(2), $ec_*([M_S^{01}])$ does not vanish and $f_{M^n_k}$ classifies the universal cover. 
We can apply Theorem \ref{dranish}. Thus $M^n_k$ is macroscopically large. 
By Lemma \ref{S}(1) we have that $M^n_k$ is rationally inessential. 

\end{proof}

\subsection{$M^n_k$ and PSC metrics.}

Motivating by the Gromov conjecture, we address the following question: does $M^n_k$ admit a PSC metric?
In this section we make a small step towards answering this question.
Namely, we prove using a result of Bolotov and Dranishnikov, that if $M^n_k$ is spin, then it does not admit a PSC metric.
Thus we support the Gromov conjecture in this case. 
We show as well, that our construction provides us with many examples of such manifolds.  

\subsubsection{Some remarks on spin structures}\label{spindisc}
Let $M$ be an oriented $n$-dimensional manifold. Let $P_{SO}(TM)$ be the principal $SO_n$ bundle associated to the tangent bundle of $M$. 
A spin structure on $M$ is a two sheeted covering of $P_{SO}(TM)$ which is connected over a fiber of $P_{SO}(TM)$. 
There may be many spin structures on $M$. 
Such a structure exists if and only if the second Stiefel-Whitney class $w_2(M)$ vanishes. 
Let $f \colon S^i \times D^{n-i} \to M$ be an embedding. 
One can always pick a framing $f$ such that a given spin structure on $M$ extends uniquely from $M \backslash f(S^i \times D^{n-i})$
to the result of the surgery with respect to $f$. 
Since $S^i \times D^{n-i}$ admits an unique spin structure for $i \neq 1$, the choice of $f$ is important only if $i=1$.

\begin{lemma}\label{smallisspin}
Let $S$ be a triangulated $n$-dimensional sphere
and $\lambda \colon S^{[0]} \to \B Z_2^{n+1}$ be a folding on a simplex characteristic function. 
Let $p \colon M_S \to C(S^*)$ be the small cover associated to $\lambda$. 
Then $w_i(M_S) = 0$ for $i>0$. In particular, $M_S$ is orientable and spin.
\end{lemma}

\begin{proof}

By \cite[Cor. 6.10, Lemma 1.14]{DJ} (compare as well \cite[Prop. 1.4]{D}) $M_S$ is stably parallelizable, 
i.e.: $TM_S \oplus \epsilon^r = \epsilon^s$ for some $r,s \geq 0$, 
where $TM_S$ is the tangent bundle of $M_S$ and $\epsilon$ is the trivial bundle. 
Let $w = \sum_{i \geq 0}w_i$ be the total Stiefel-Whitney class.
In the following we use the Whitney product formula and the fact that $w(\epsilon^s) = 1 \in H^0(M_S;\B Z_2)$:
$$ 
w(\epsilon^r) = w(TM_S \oplus \epsilon^s) = w(TM_S) \cup w(\epsilon^s) = w(TM_S).
$$
By definition $w(M_S) = w(TM_S)$, thus $w_i(M_S) = w_i(\epsilon^r) = 0$ for $i>0$.
\end{proof}

\begin{corollary} 
Assume that the surgery we use in Step 3 of the construction allows us to extend spin structures.
Then $M^n_k$ is spin.
\end{corollary} 

\begin{proof}
From Lemma \ref{smallisspin} follows that $M_S$ is spin. 
A manifold $N = M_S^0 \# (-M_S^1)$ is spin as a connected sum of spin manifolds (surgery of index $1$).
The manifold $M_k^n$ is the result of the surgery on $N$, which was arranged such that $M_k^n$ is spin.
\end{proof}
 
\subsubsection{Positive scalar curvature}

The crucial result which we use here is a theorem due to D.~Bolotov and A.~Dranishnikov

\begin{theorem}{\cite[Col.4.4]{MR2578547}}\label{psc}
The Gromov Conjecture holds for spin n-manifolds $M$, having the cohomological dimension $cd(\pi_1(M)) \leq n+3$, and satisfying the Strong Novikov Conjecture.
\end{theorem}

\begin{theorem}
If $M^n_k$ is spin, then it does not admit a Riemannian metric of positive scalar curvature. 
\end{theorem}

\begin{proof}
We check that the assumptions of Theorem \ref{psc} are satisfied. 
It is well known that subgroups of Coxeter groups satisfy the Baum-Connes conjecture, which implies the Strong Novikov Conjecture. 
The inequality for $cd(\pi)$ follows from the fact that $M_L$ is a classifying space of $\pi$, so $cd(\pi) \leq dim(M_L) = n+1$.
We already know that $M_k^n$ is macroscopically large, thus by Theorem \ref{psc}, $M_k^n$ can not admit a metric of positive scalar curvature. 
\end{proof}

\section{Further examples}

In this section we describe a construction of rationally inessential macroscopically large manifolds
which generalizes that from Section \ref{construction}.
Instead of working with small covers, we start with a Davis complex.
Then we find an appropriate subgroups of right angled Coxeter groups and pass to quotients.

Let $X$ be a simplicial complex. 
By $W_X$ we denote the right angled Coxeter group associated to $X$ (as in \ref{properties})
and by $\Sigma_X$ its Davis complex. 
By $g \in W_X$, $l(g)$ we denote the minimal number of generators one needs to express $g$.  
Let $W_X^+ < W_X$ be the subgroup of elements whose Coxeter length is even.
We assume all complexes to be flag and associated Coxeter groups to be infinite. 
The usual fundamental domain of the action of $W_X$ on $\Sigma_X$ is homeomorphic to $C(X)$ and is called the Davis cell.

Let $S < L$ be a pair of compact simplicial complexes. 
We assume that $S$ is an oriented null-bordant manifold such that $[S] \in H(L;\B Z)$ is a nontrivial $k$-torsion class.
Moreover, $S$ can be of codimension bigger than one in $L$, but we assume that $S$ is at least $3$-dimensional. 
Let $D$ be a simplicial chain in $L$ such that $\partial D = kS$. 

The inclusion $S < L$ induces an inclusion on the level of Davis complexes: $\Sigma_S < \Sigma_L$. 
Let $\Gamma_S$ and $\Gamma_L$ be finite index torsion-free subgroups of $W_S^+$ and $W_L^+$ 
such that $\Gamma_S < \Gamma_L$ and $\Gamma_L \cap W_S = \Gamma_S$. 
E.g. $\Gamma_S$ and $\Gamma_L$ can be taken to be the derived subgroups of $W_S$ and $W_L$, respectively
(the derived subgroup of a right angled Coxeter group is torsion free, the proof is analogous to that of Fact \ref{factsmall}(1)). 
Then $\Gamma_S$ and $\Gamma_L$ act freely, orientation preserving and cocompactly on $\Sigma_S$ and $\Sigma_L$, respectively. 
The group $\Gamma_L$ in this construction plays the role of $\pi_1(M_L)$.

Notice that $\Sigma_L$ considered as a $\Gamma_S$ space, is a classifying bundle of $\Gamma_S$. 
Since $\Sigma_S < \Sigma_L$ is $\Gamma_S$-invariant
(here we use the assumption that $\Gamma_S < W_S^+$), 
it defines a class $[\Sigma_S/\Gamma_S] \in H_*(\Sigma_L/\Gamma_S;\B Z) = H_*(\Gamma_S;\B Z)$.

\begin{lemma}\label{manrep}
The class $[\Sigma_S/\Gamma_S] \in H_*(\Gamma_S;\B Z)$ can be represented as a pushforward of the fundamental class of a manifold.
\end{lemma}  

\begin{proof}
First, we consider the Davis cell $C(S)$. By the assumption $S$ is null-bordant, thus there exists a manifold $B$ such that $S = \partial B$. 
The link of the apex of the cone $C(S)$ is $S$, thus we can truncate the apex and glue in the manifold $B$, getting rid of the singularity. 
Now we take care of $\Sigma_S$. 
The only points of $\Sigma_S$ which have noneuclidean neighborhoods are apexes of translates of the Davis cell.
The group $\Gamma_S$ acts on $\Sigma_S$ freely and cocompactly.  
Thus the quotient $\Sigma_S/\Gamma_S < \Sigma_L/\Gamma_S$ is compact and is a manifold except apexes of cones. 
We can do the above surgery for every apex of $\Sigma_S/\Gamma_S$. We obtain a manifold, denote it by $M_S$, 
together with a map $g \colon M_S \to \Sigma_L/\Gamma_S$, which collapses just glued copies of $B$ again to apexes of appropriate cones.
We have that $g_*([M_S]) = [\Sigma_S/\Gamma_S] \in H_*(\Gamma_S;\B Z)$.
\end{proof}

Define a chain $\alpha = \Sigma_{g \in W_L} (-1)^{l(g)}g.\Sigma_S$.
It is a $\Gamma_L$-equivariant chain because $\Gamma_L < W_L^+$. Thus it
defines a class $\beta = [\alpha / \Gamma_L] \in H_*(\Gamma_L;\B Z)$.
Since $\Gamma_L \cap W_S = \Gamma_S$, the class $\beta$ is a finite disjoint sum of $\pm \Sigma_S / \Gamma_S$ and it plays the role of $M^{01}_S$. 
The number of components in this sum equals $[W_L \colon \Gamma_L] / [W_S \colon \Gamma_S]$. 
Indeed, $\Sigma_S / \Gamma_S$ consists of $[W_S \colon \Gamma_S]$ cones $C(S)$ 
and $\alpha / \Gamma_L$ consists of $[W_L \colon \Gamma_L]$ cones (in Section \ref{construction} it was $2 = 2^{n+1}/2^n$).
By Lemma \ref{manrep}, $\beta$ is represented by a connected sum of manifolds $\pm M_S$. Denote this connected sum by $N$.
The rest of the construction is the surgery procedure described in Step 3 of \ref{construction}.
We call the resulting manifold $M(L,S)$. 

\begin{theorem}
The manifold $M(L,S)$ is rationally inessential and macroscopically large. 
\end{theorem}

\begin{proof}
The proof of Lemma \ref{S} goes through essentially without changes.  
Namely, the assumption that $kS$ is the boundary of $D$ allows to carry out the computations 
as in Lemma \ref{S}(1). The only difference is that we replace $q_*(M_D)$ with the lift of $D$ to $\Sigma_L/\Gamma_L$. 
That is, with $\sum_{r \in R}(-1)^{l(r)}r.C(D)$, where $R$ is a set of representatives of the cosets of $\Gamma_L$ in $W_L$. 
As well, Lemma \ref{lift} admits a straightforward generalisation to the situation
when we lift a chain to a cover of $C(L)$ of the form $\Sigma_L/\Gamma_L$. 
In Lemma \ref{S}(2) we really work with the universal cover of $M_L$, which is the Davis complex for $L$. 
\end{proof}

\begin{remark}
The advantage of using small covers to construct our examples, except their intrinsic beauty,
gives us a better insight into the (co)homology ring of $M^n_k$ and possible spin structures.
They are the simplest possible examples to deal with.  
\end{remark}

This paper leaves the following question open.

\begin{question*}
Does $M(L,S)$ admit a PSC metric? In particular, does $M_k^n$ with no additional assumptions on Step 3, see \ref{spindisc}, admit a PSC metric?
\end{question*}

\subsection*{Acknowledgements}

I wish to thank Tadeusz Januszkiewicz for his constant help and guidance throughout this project
and Alexander N. Dranishnikov for helpful discussions and comments.
I am also indebted to \'{S}wiatos\l{}aw Gal and \L ukasz Garncarek for their help in improving the presentation of the paper, 
and the anonymous referee for pointing out simpler proof of Lemma \ref{smallisspin}. 

\bibliography{../../../bib/bibliography}

\def\cprime{$'$}
\begin{thebibliography}{10}

\bibitem{MR1145337}
{\sc Block, J., and Weinberger, S.}
\newblock Aperiodic tilings, positive scalar curvature and amenability of
  spaces.
\newblock {\em J. Amer. Math. Soc. 5}, 4 (1992), 907--918.

\bibitem{MR2578547}
{\sc Bolotov, D.~V., and Dranishnikov, A.~N.}
\newblock On {G}romov's scalar curvature conjecture.
\newblock {\em Proc. Amer. Math. Soc. 138}, 4 (2010), 1517--1524.

\bibitem{D}
{\sc Davis, M.~W.}
\newblock Some aspherical manifolds.
\newblock {\em Duke Math. J. 55}, 1 (1987), 105--139.

\bibitem{dlf}
{\sc Davis, M.~W.}
\newblock The cohomology of a {C}oxeter group with group ring coefficients.
\newblock {\em Duke Math. J. 91}, 2 (1998), 297--314.

\bibitem{Davis}
{\sc Davis, M.~W.}
\newblock {\em The geometry and topology of {C}oxeter groups}, vol.~32 of {\em
  London Mathematical Society Monographs Series}.
\newblock Princeton University Press, Princeton, NJ, 2008.

\bibitem{DJ}
{\sc Davis, M.~W., and Januszkiewicz, T.~L.}
\newblock Convex polytopes, {C}oxeter orbifolds and torus actions.
\newblock {\em Duke Math. J. 62}, 2 (1991), 417--451.

\bibitem{dranish.nowa}
{\sc Dranishnikov, A.~N.}
\newblock On macroscopic dimension of universal covering of closed manifolds.
\newblock {\em Tr. Mosk. Mat. Obs. 74:2\/} (2013), 279–296.

\bibitem{dranish.conjecture}
{\sc Dranishnikov, A.~N.}
\newblock On {G}romov's positive scalar curvature conjecture for duality
  groups.
\newblock {\em J. Topol. Anal. 6}, 3 (2014), 397--419.

\bibitem{MR2365352}
{\sc Geoghegan, R.}
\newblock {\em Topological methods in group theory}, vol.~243 of {\em Graduate
  Texts in Mathematics}.
\newblock Springer, New York, 2008.

\bibitem{MR720933}
{\sc Gromov, M., and Lawson, Jr., H.~B.}
\newblock Positive scalar curvature and the {D}irac operator on complete
  {R}iemannian manifolds.
\newblock {\em Inst. Hautes \'Etudes Sci. Publ. Math.}, 58 (1983), 83--196
  (1984).

\bibitem{MR1389019}
{\sc Gromov, M.~L.}
\newblock Positive curvature, macroscopic dimension, spectral gaps and higher
  signatures.
\newblock In {\em Functional analysis on the eve of the 21st century, {V}ol.\
  {II} ({N}ew {B}runswick, {NJ}, 1993)}, vol.~132 of {\em Progr. Math.}
  Birkh\"auser Boston, Boston, MA, 1996, pp.~1--213.

\bibitem{MR2986138}
{\sc Nowak, P.~W., and Yu, G.}
\newblock {\em Large scale geometry}.
\newblock EMS Textbooks in Mathematics. European Mathematical Society (EMS),
  Z\"urich, 2012.

\end{thebibliography}
\bibliographystyle{acm}
 
\end{document}